\documentclass[12pt]{article}
\usepackage[T1]{fontenc}
\usepackage[utf8]{inputenc}
\usepackage{lmodern}
\usepackage{fullpage}
\usepackage{amsmath,amssymb,amsthm,mathtools}
\usepackage{microtype}
\usepackage{booktabs,array}
\usepackage{xcolor}
\usepackage{enumitem}
\usepackage{float}

\usepackage{hyperref}
\setlist[enumerate]{itemsep=0.3em,topsep=0.4em}

\newcommand{\seqnum}[1]{\href{https://oeis.org/#1}{\rm \underline{#1}}}

\theoremstyle{plain}
\newtheorem{theorem}{Theorem}

\newtheorem{lemma}[theorem]{Lemma}
\newtheorem{proposition}[theorem]{Proposition}

\theoremstyle{definition}

\theoremstyle{remark}

\newcommand{\Enn}{\mathbb{N}}
\newcommand{\Zee}{\mathbb{Z}}

\newcommand{\Par}{\psi}
\newcommand{\ii}{\mathrm{i}}
\newcommand{\vTwo}{\nu_2}

\newcommand{\word}{\mathbf{a}}

\title{An infinite walk in $\Enn^{16}$, using only unit steps,
with no three collinear points}

\author{Jeffrey Shallit\\
School of Computer Science\\
University of Waterloo\\
Waterloo, ON  N2L 3G1\\
Canada\\
\href{mailto:shallit@uwaterloo.ca}{\tt shallit@uwaterloo.ca}}

\begin{document}
\maketitle

\begin{abstract}
Is it possible to walk to infinity, avoiding three
collinear points, using as steps only the standard unit basis
vectors $(0,0,\ldots, 0, 1, 0, \ldots, 0)$?  We prove this
is possible in $16$ dimensions.  
\end{abstract}

The results in this paper have now been superseded by the much
stronger results in \url{https://arxiv.org/abs/2609.20366}.

\section{Introduction}
Let $k \geq 2$, and consider an infinite
walk in $\Enn^k$ using only steps that are
basis vectors of the form $(0,0,\ldots, 0,1,0,\ldots, 0)$.  For which 
$k$ can one avoid some finite number of collinear points?

For dimension $k = 2$ every such infinite planar walk must have every possible
finite number of collinear points.
This is the content of the famous Gerver-Ramsey theorem 
\cite{Bro71a,Mo72,GR79}.  Also see \cite{Ko26a}.

For dimension
$k = 3$ Gerver and Ramsey proved that there is an infinite walk that avoids
$5^{11}+1$ collinear points.  This was recently improved to
$189$ by Lidbetter \cite{Lid24}.

Of course it is impossible to avoid $2$ collinear points in such a
walk.  This
raises the natural question:  for which dimensions
$k$ can one avoid three collinear
points?  One cannot do it for $k = 3$, as shown by Brown
\cite{Bro71a}, since every walk of $8$ unit steps contains
three collinear points.    

In this paper we solve the problem for $k = 16$.  The construction can be
given explicitly as a morphism on $16$ letters.

The description of such a walk can be given picturesquely as follows:  if
an infinitesimal obstruction is placed at each lattice point visited, then
it is possible to see all the other lattice points of the walk from every
point of the walk.

\subsection{A reformulation}

The problem of avoiding $3$ collinear points
can be rephrased in an equivalent way as follows, which interprets
the original problem in the terminology of combinatorics on words.

Consider an infinite sequence $(a_n)_{n \geq 0}$ over a finite alphabet $\Sigma_k = \{ 0,1,\ldots, k-1 \}$, and call it {\it good\/} if it contains no two consecutive nonempty blocks $x,y$ (possibly of different sizes) with the property that the $|x|_a / |x| = |y|_a / |y|$ for all letters $a \in \Sigma_k$.  Here $|x|$ is the length of the block and $|x|_a$ is the number of occurrences of a letter $a$.  

The vector $\Par(x) := (|x|_0, |x|_1, \ldots, |x|_{k-1})$ is often called
the {\it Parikh vector} of the word $x \in \{ 0,1,\ldots, k-1\}^*$.
The vector of letter frequencies
$\Par(x)/|x|$ is called the normalized Parikh vector.  Thus a good
sequence has no two consecutive nonempty blocks with the same normalized
Parikh vector.

Equality of letter frequencies is sometimes called \emph{weak abelian
equivalence}; see the discussion 
in~\cite{AP16} and ~\cite[Section 8.2]{FP23}.
Accordingly, the forbidden factor $xy$ is sometimes
termed a {\it weak abelian square}.
(An ordinary \emph{abelian square} additionally has $|x|=|y|$ and consists
of two consecutive words, where the second is a permutation of the first.
Avoiding ordinary abelian squares alone is therefore
not the whole requirement in this problem.)

\begin{proposition}
\label{prop:walk}
Let $\Sigma=\{0,\ldots,k-1\}$ and let $E_0,\ldots,E_{k-1}$ be the
standard basis vectors of $\Enn^k$.  Define
\[
  P_0=0,\qquad P_n=\sum_{j=0}^{n-1}E_{b_j}\quad(n\geq1).
\]
The word $(b_n)_{n\geq0}$ is good if and only if no three distinct
vertices of $(P_n)_{n\geq0}$ are collinear.
\end{proposition}

\begin{proof}
The sum of the coordinates of $P_n$ is $n$, so all vertices are
distinct.  Fix indices $\alpha<\beta<\gamma$ and set
$p=\beta-\alpha>0$, $q=\gamma-\beta>0$.  With $x,y$ as in the
definition,
\[
  P_\beta-P_\alpha=\Par(x),\qquad
  P_\gamma-P_\beta=\Par(y).
\]
If the three vertices are collinear, their two displayed nonzero
displacements are scalar multiples: $\Par(y)=\lambda\Par(x)$.
Taking coordinate sums gives $q=\lambda p$, so $\lambda=q/p>0$.
Consequently $\Par(x)/p=\Par(y)/q$.  Conversely, this equality makes
the two displacements positive scalar multiples, so the three vertices
are collinear.  This holds for every ordered triple of indices, proving
both directions.
\end{proof}

For avoidance of triples, Cambie and Kalviainen~\cite{CK26} recently gave a
different construction in $\Zee^3$ with sixteen possible step vectors, 
but they are not all unit vectors.  Our construction is a small
variation on theirs.  Also see \cite{Ko26b}, which gave a lower bound
on the length of the longest good sequence in dimension $k$ (now made
obsolete for $k \geq 16$).

\section{The construction}

We will show that the sequence
$${\bf b} = 1, 5, 6, 9, 6, 10, 11, 13, 6, 10, 11, 14, 11, 15, 12, 1, 6, 10, 11, 14, 11, 15, 12, 2, 11, 15, 12, 3, 12, 0, 1, 5,\ldots$$
is good.  This sequence can be defined as the fixed point of the following
morphism:
\begin{table}[H]
\begin{center}
\begin{tabular}{@{}cc@{\qquad\qquad}cc@{}}
\toprule
Letter & Image & Letter & Image\\
\midrule
$0$ & $1\;4$ & $8$  & $11\;12$\\
$1$ & $1\;5$ & $9$  & $11\;13$\\
$2$ & $1\;6$ & $10$ & $11\;14$\\
$3$ & $1\;7$ & $11$ & $11\;15$\\
$4$ & $6\;8$ & $12$ & $12\;0$\\
$5$ & $6\;9$ & $13$ & $12\;1$\\
$6$ & $6\;10$ & $14$ & $12\;2$\\
$7$ & $6\;11$ & $15$ & $12\;3$\\
\bottomrule
\end{tabular}
\end{center}
\caption{The morphism generating the good sequence $\bf b$.}
\label{tab1}
\end{table}

For $n\in\Enn$, let $s_2(n)$ denote the number of $1$s in the binary
expansion of $n$, and let $t_n=s_2(n)\pmod 4$.  This is sequence
\seqnum{A179868} in the OEIS \cite{OEIS}.
Set
\begin{equation}
  \Gamma=\{0,1,2,3\}^2,
  \qquad a_n=(t_n,t_{n+1}),
  \qquad \word=a_0a_1a_2\cdots.
  \label{eq:word}
\end{equation}
Each ordered pair is \emph{one letter} of the sixteen-letter alphabet
$\Gamma$.  For example,
\[
  \word=(0,1)(1,1)(1,2)(2,1)(1,2)(2,2)(2,3)(3,1)\cdots.
\]
Now form the desired good sequence
$$ {\bf b} = 1, 5, 6, 9, 6, 10, 11, 13, 6, 10, 11, \ldots$$
by applying the coding
\begin{equation}
  \rho(r,s)=4r+s
  \label{eq:relabel}
\end{equation}
to each element of $\word$.  

Appending a binary digit gives
\[
  s_2(2n)=s_2(n),\qquad s_2(2n+1)=s_2(n)+1.
\]
Hence
\begin{equation}
  t_{2n}=t_n,\qquad t_{2n+1}=t_n+1\pmod4.
  \label{eq:state-recursion}
\end{equation}
For $r\in\{0,1,2,3\}$, write $r^+$ for the representative of
$r+1\pmod4$ in that set.

\begin{proposition}
\label{prop:morphism}
The word $\word$ is the fixed point starting with $(0,1)$ of the
$2$-uniform morphism
\begin{equation}
  H(r,s)=(r,r^+)(r^+,s).
  \label{eq:morphism}
\end{equation}
\end{proposition}

\begin{proof}
A morphism acts on words by replacing each letter with its image and
concatenating the resulting images.  Each image in
\eqref{eq:morphism} consists of two letters, so $H$ is $2$-uniform.
By \eqref{eq:state-recursion},
\begin{align*}
 a_{2n}&=(t_{2n},t_{2n+1})=(t_n,t_n^+),\\
 a_{2n+1}&=(t_{2n+1},t_{2n+2})=(t_n^+,t_{n+1}).
\end{align*}
Thus $H(a_n)=a_{2n}a_{2n+1}$ for every $n$, and concatenation gives
$H(\word)=\word$.

Also $t_0=0$, $t_1=1$, and
$H(0,1)=(0,1)(1,1)$.  Therefore $H^{j+1}(0,1)$ begins with
$H^j(0,1)$, whose length is $2^j$.  These prefixes determine
a unique infinite word.  
Since $\word$
is such a fixed point, it is precisely this limit.
\end{proof}

Finally, applying the coding $\rho$ rewrites the morphism as
\begin{equation}
  (4r+s) \rightarrow (4r+r^+)\,(4r^++s),
  \label{eq:numeric-morphism}
\end{equation}
which gives us the morphism in Table~\ref{tab1}.

\section{The arithmetic ingredient}
\label{sec:arithmetic}

For a nonzero integer $v$, define the $2$-adic valuation $\vTwo(v)$ by writing
$v=2^e w$, with $w$ odd. Then $\vTwo(v)=e$.  For a positive rational
number $a/b$, define
\[
  \vTwo(a/b)=\vTwo(a)-\vTwo(b).
\]
This is independent of the chosen representation, because
$a/b=c/d$ implies $ad=bc$, and the exponent of $2$ in a product is
the sum of the exponents in its factors.  In particular,
\begin{equation}
  \vTwo(xy)=\vTwo(x)+\vTwo(y),\qquad
  \vTwo(x/y)=\vTwo(x)-\vTwo(y).
  \label{eq:valuation-rules}
\end{equation}

Let $\ii^2=-1$, and put
\begin{equation}
 u_n=\ii^{t_n}=\ii^{s_2(n)},\qquad
 Z_0=0,\qquad Z_n=\sum_{j=0}^{n-1}u_j\quad(n\geq1).
 \label{eq:gaussian-walk}
\end{equation}
The equality of the two powers of $\ii$ uses $\ii^4=1$.  Each $u_j$
is one of $1,\ii,-1,-\ii$, so each $Z_n$ is a Gaussian integer.
For a Gaussian integer $X+\ii Y$, its squared modulus is the ordinary
nonnegative integer $X^2+Y^2$.

Equation~\eqref{eq:state-recursion} gives
\[
  u_{2n}=u_n,\qquad u_{2n+1}=\ii u_n.
\]
Pairing consecutive terms in the defining sum for $Z_{2n}$, we obtain
\begin{align*}
 Z_{2n}
   &=\sum_{j=0}^{n-1}(u_{2j}+u_{2j+1})
     =\sum_{j=0}^{n-1}(1+\ii)u_j
     =(1+\ii)Z_n,\\
 Z_{2n+1}&=Z_{2n}+u_{2n}=(1+\ii)Z_n+u_n.
\end{align*}
Equivalently, for $\varepsilon\in\{0,1\}$,
\begin{equation}
 u_{2n+\varepsilon}=\ii^\varepsilon u_n,
 \qquad Z_{2n+\varepsilon}=(1+\ii)Z_n+\varepsilon u_n.
 \label{eq:gaussian-recursion}
\end{equation}

The following lemma is from \cite{CK26}.
\begin{lemma}
\label{lem:valuation}
If $0\leq m<n$ and $t_m=t_n$, then $Z_n-Z_m\ne0$ and
\begin{equation}
  \vTwo\bigl(|Z_n-Z_m|^2\bigr)=\vTwo(n-m).
  \label{eq:key-valuation}
\end{equation}
\end{lemma}

\begin{proof}
Let $e=\vTwo(n-m)$.  We first explain one halving step.  If $n-m$
is even, the indices have the same parity, so write
\[
  m=2a+\varepsilon,\qquad n=2b+\varepsilon,
  \qquad \varepsilon\in\{0,1\}.
\]
The hypothesis $t_m=t_n$ is equivalent to $u_m=u_n$.  The first
identity in \eqref{eq:gaussian-recursion} then gives
$\ii^\varepsilon u_a=\ii^\varepsilon u_b$, hence $u_a=u_b$ and
$t_a=t_b$.  The second identity consequently gives
\begin{align*}
 Z_n-Z_m
  &=(1+\ii)(Z_b-Z_a)+\varepsilon(u_b-u_a)\\
  &=(1+\ii)(Z_b-Z_a).
\end{align*}

After exactly $e$ such steps, there are indices $m'<n'$ with odd
difference and
\begin{equation}
  Z_n-Z_m=(1+\ii)^e(Z_{n'}-Z_{m'}).
  \label{eq:halving}
\end{equation}
If $e=0$, this means $m'=m$ and $n'=n$.

Write $Z_{n'}-Z_{m'}=X+\ii Y$.  This difference is a sum of the odd
number $n'-m'$ of Gaussian units $u_j$.  The sum of the real and
imaginary coordinates of each unit is either $1$ or $-1$, both odd.
It follows that $X+Y$ is odd.  Since an integer and its square have
the same parity,
\[
  X^2+Y^2\equiv X+Y\equiv1\pmod2.
\]
In particular, this squared modulus is positive and odd.  Taking
squared moduli in \eqref{eq:halving}, and using $|1+\ii|^2=2$, gives
\[
  |Z_n-Z_m|^2=2^e(X^2+Y^2).
\]
The right side is nonzero and has exponent of $2$ equal to $e$.
This proves both assertions.
\end{proof}

For a letter $(r,s)\in\Gamma$ and a finite block $x$ of $\word$, now write
\[
  C_x(r,s)=|x|_{(r,s)}.
\]
Let $e_0,e_1,e_2,e_3$ be the standard basis of $\Enn^4$.
To a pair $(r,s)$ assign the vector $e_r-e_s$.

For indices $\alpha<\beta$ and the block
$x=a_\alpha\cdots a_{\beta-1}$, summing these vectors gives
\begin{align}
 \sum_{r,s=0}^3 C_x(r,s)(e_r-e_s)
   &=\sum_{j=\alpha}^{\beta-1}(e_{t_j}-e_{t_{j+1}})\notag\\
   &=e_{t_\alpha}-e_{t_\beta}.
   \label{eq:telescoping}
\end{align}

\begin{lemma}
\label{lem:balance}
Suppose consecutive blocks
$x=a_\alpha\cdots a_{\beta-1}$ and
$y=a_\beta\cdots a_{\gamma-1}$, with $\alpha<\beta<\gamma$, have the
same normalized Parikh vectors.   Then
$ t_\alpha=t_\beta=t_\gamma$.
\end{lemma}

\begin{proof}
Set $p=\beta-\alpha$ and $q=\gamma-\beta$.  The hypothesis says
\begin{equation}
  qC_x(r,s)=pC_y(r,s)\qquad(0\leq r,s\leq3).
  \label{eq:scaled-counts}
\end{equation}
Multiply each equality by $e_r-e_s$, sum, and apply
\eqref{eq:telescoping} to the two blocks.  The result is
\[
  q(e_{t_\alpha}-e_{t_\beta})
   =p(e_{t_\beta}-e_{t_\gamma}),
\]
which rearranges to
\begin{equation}
  q e_{t_\alpha}+p e_{t_\gamma}=(p+q)e_{t_\beta}.
  \label{eq:convex-balance}
\end{equation}
Look at the coordinate labeled $t_\beta$.  Its value on the left is
\[
 q\,\mathbf{1}_{\{t_\alpha=t_\beta\}}
 +p\,\mathbf{1}_{\{t_\gamma=t_\beta\}},
\]
and its value on the right is $p+q$.  Because $p$ and $q$ are both
strictly positive, equality requires both indicators to be $1$.
\end{proof}

\section{Proof that the sequence is good}
\label{sec:goodness}

\begin{theorem}
\label{thm:good}
The infinite word $\word$ defined by \eqref{eq:word} is good.
Equivalently, the standard-basis walk in $\Enn^{16}$ obtained from
$\rho(a_n)=4t_n+t_{n+1}$ has no three collinear vertices.
\end{theorem}

\begin{proof}
Assume, to get a contradiction, that consecutive nonempty blocks $x,y$
have identical normalized Parikh vectors.  Write their positions as
$[\alpha,\beta)$ and $[\beta,\gamma)$, and set
$p=\beta-\alpha>0$, $q=\gamma-\beta>0$.  Equation
\eqref{eq:scaled-counts} holds, and Lemma~\ref{lem:balance} gives
\begin{equation}
  t_\alpha=t_\beta=t_\gamma.
  \label{eq:same-states}
\end{equation}

Assign $(r,s)$ the complex number $\ii^r$.  Since the first coordinate of
$a_j$ is $t_j$, grouping gives
\begin{align*}
 \sum_{r,s=0}^3 C_x(r,s)\ii^r
   &=\sum_{j=\alpha}^{\beta-1}\ii^{t_j}
     =Z_\beta-Z_\alpha,\\
 \sum_{r,s=0}^3 C_y(r,s)\ii^r
   &=\sum_{j=\beta}^{\gamma-1}\ii^{t_j}
     =Z_\gamma-Z_\beta.
\end{align*}
Consequently, multiplying \eqref{eq:scaled-counts} by $\ii^r$ and
summing yields
\[
  q(Z_\beta-Z_\alpha)=p(Z_\gamma-Z_\beta).
\]
Put $X=Z_\beta-Z_\alpha$ and $Y=Z_\gamma-Z_\beta$.
There is a common complex number $V=X/p=Y/q$.  Adding the identities
$X=pV$, $Y=qV$ shows that
\begin{equation}
 \frac{X}{p}=\frac{Y}{q}=\frac{X+Y}{p+q}=V.
 \label{eq:slopes}
\end{equation}

By \eqref{eq:same-states}, Lemma~\ref{lem:valuation} applies to each
of the three pairs of indices $(\alpha,\beta)$,
$(\beta,\gamma)$, and $(\alpha,\gamma)$.  Therefore
\begin{align}
 \vTwo(|X|^2)&=\vTwo(p),\notag\\
 \vTwo(|Y|^2)&=\vTwo(q),\notag\\
 \vTwo(|X+Y|^2)&=\vTwo(p+q).
 \label{eq:three-valuations}
\end{align}
The same lemma also proves that all three displacements are nonzero.
In particular $|V|^2$ is a positive rational number, so its valuation
is defined.

Using the first expression for $V$ in \eqref{eq:slopes}, followed by
\eqref{eq:valuation-rules} and \eqref{eq:three-valuations}, gives
\begin{align*}
 \vTwo(|V|^2)
  &=\vTwo\left(\frac{|X|^2}{p^2}\right)
    =\vTwo(|X|^2)-2\vTwo(p)\\
  &=\vTwo(p)-2\vTwo(p)=-\vTwo(p).
\end{align*}
Using the other two expressions for $V$ in exactly the same way gives
\[
  \vTwo(|V|^2)=-\vTwo(q)=-\vTwo(p+q).
\]
Thus
\begin{equation}
  \vTwo(p)=\vTwo(q)=\vTwo(p+q).
  \label{eq:impossible}
\end{equation}
Let the common value of the first two terms be $d$.  Then
$p=2^d p_0$ and $q=2^d q_0$ for positive odd integers $p_0,q_0$.
Their sum is even, so
\[
  p+q=2^d(p_0+q_0)
  \quad\Longrightarrow\quad \vTwo(p+q)\geq d+1,
\]
contradicting \eqref{eq:impossible}.  No such blocks $x,y$ can exist.

Relabeling letters by the bijection $\rho$ preserves all frequency
equalities.  Proposition~\ref{prop:walk} now gives the asserted
geometric conclusion.
\end{proof}

\section{The dimension bound and the proposed smaller morphisms}
\label{sec:bounds}

Define $k_{\min}$ to be the least size of an alphabet admitting a good
infinite word.  Theorem~\ref{thm:good} proves that this minimum exists
and is at most $16$.  

As mentioned above, we know that $k_{\min} \geq 4$.  Therefore
it is still of interest to determine its exact value.

I found a number of candidate cyclic morphisms, but have been unable to prove
that they work.  I list them here in the hope that they may aid
other researchers to improve the bound $k_{\min} \leq 16$.

\begin{center}
\begin{tabular}{@{}ccc@{}}
\toprule
Alphabet size $k$ & Proposed image of $0$ & Image length\\
\midrule
$5$ & $01213101314310$ & $14$\\
$6$ & $01210140$       & $8$\\
$7$ & $012520$         & $6$\\
$8$ & $012560$         & $6$\\
\bottomrule
\end{tabular}
\end{center}
For these proposals, the image of letter $r$ is obtained by adding
$r$ modulo $k$ to every letter in the displayed image of $0$.

\section{Declaration of AI usage}

I originally proposed this problem about walks on unit vectors in June 2018,
and mentioned it to many people.  The recent construction of Cambie
and Kalviainen, and the release of GPT-6 Astra, led me to try this new LLM
on the unit vector version of the problem.
Much of this paper was written by ChatGPT 6 Astra.

\end{document}